\documentclass[10pt,reqno]{amsart}
\usepackage{amssymb,array,bm,hyperref,mathdots,setspace,tikz}
\usepackage[vcentermath,enableskew]{youngtab}
\usetikzlibrary{matrix}

\renewcommand{\aa}{{\bm a}}
\newcommand{\bb}{{\bm b}}
\newcommand{\BB}{\mathcal{B}}
\newcommand{\bbox}{\operatorname{box}}
\newcommand{\CC}{\mybb{C}}
\newcommand{\ii}{{\mathbf i}}

\newcommand{\mybb}[1]{\mathbf{#1}}
\newcommand{\non}{\operatorname{non}}

\newcommand{\wt}{{\rm wt}}
\newcommand{\zz}{{\bm z}}
\newcommand{\ZZ}{\mybb{Z}}

\theoremstyle{plain}
\newtheorem{thm}{Theorem}[section]
\newtheorem{lemma}[thm]{Lemma}

\newtheorem{prop}[thm]{Proposition}

\theoremstyle{definition}
\newtheorem{dfn}[thm]{Definition}
\newtheorem{ex}[thm]{Example}
\newtheorem{remark}[thm]{Remark}
\numberwithin{equation}{section}
\numberwithin{figure}{section}
\numberwithin{table}{section}
\begin{document}
\title{Combinatorics of the Casselman-Shalika formula in type $A$}
\author{Kyu-Hwan Lee}
\address{Department of Mathematics \\ University of Connecticut \\  Storrs, CT 06269-3009, U.S.A. and Korea Institute for Advanced Study, Seoul, Korea}
\email{khlee@math.uconn.edu}
\author{Philip Lombardo}
\address{Department of Mathematics and Computer Science \\ St.\ Joseph's College \\ Patchogue, NY 11772, U.S.A.}
\email{plombardo@sjcny.edu}
\author{Ben Salisbury}
\address{Department of Mathematics \\ University of Connecticut \\ Storrs, CT 06269-3009, U.S.A.}
\email{benjamin.salisbury@uconn.edu}

\maketitle

\begin{abstract}
In the recent works of Brubaker-Bump-Friedberg, Bump-Nakasuji, and others, the product in the Casselman-Shalika formula is written as a sum over a crystal. The coefficient of each crystal element is defined using the data coming from the whole crystal graph structure. In this paper, we adopt the tableaux model for the crystal and obtain the same coefficients using data from each individual tableaux; i.e., we do not need to look at the graph structure.  We also show how to combine our results  with tensor products of crystals to obtain the sum of coefficients for a given weight. The sum is a $q$-polynomial which exhibits many interesting properties. We use examples to illustrate these properties.

\end{abstract}


\section{Introduction}

Let $F$ be a $p$-adic field with ring of integers $\mathfrak{o}_F$ and residue field of size $q$. We denote by $\varpi$ a uniformizer of $\mathfrak o_F$. Suppose $N^-$ is the maximal unipotent subgroup of $\operatorname{GL}_{r+1}(F)$ with maximal torus $T$, and $f^\circ$ denotes the standard spherical vector corresponding to an unramified character $\chi$ of $T$. Let $T(\CC)$ be the maximal torus in the  dual group $\operatorname{GL}_{r+1}(\CC)$ of $\operatorname{GL}_{r+1}(F)$, and let $\zz \in T(\CC)$ be the element corresponding to $\chi$ via the Satake isomorphism.

For a dominant integral weight $\lambda = (\lambda_1\ge\lambda_2\ge\cdots\ge\lambda_{r+1})$, we define
\[
\psi_\lambda\left(\begin{array}{cccc}
1 & & & \\
x_{2,1} & 1 & & \\
\vdots & & \ddots & \\
x_{r+1,1} & \cdots & x_{r+1,r} & 1
\end{array}\right) = \psi_0(\varpi^{\lambda_1-\lambda_2}x_{r+1,r} + \cdots + \varpi^{\lambda_r-\lambda_{r+1}}x_{2,1}),
\]
where $\psi_0$ is a fixed additive character on $F$ which is trivial on $\mathfrak{o}_F$ but not on $\mathfrak{p}^{-1}$.  Let $\chi_\lambda(\zz)$ be the irreducible character of $\operatorname{GL}_{r+1}(\CC)$ with highest weight $\lambda$.  Then the Casselman-Shalika formula is
\begin{equation}\label{eq:CS}
\int_{N^-(F)} f^\circ({\bm n})\psi_\lambda({\bm n}) \,\mathrm{d}{\bm n} = \prod_{\alpha>0}
(1-q^{-1}\zz^{\alpha})\chi_\lambda(\zz).
\end{equation}

Recently, in the works \cite{BBF:11,BN:10} of Brubaker-Bump-Friedberg and Bump-Nakasuji, the product is written as a sum over the crystal $\BB(\lambda+\rho)$. (See
also \cite{McN:11}.) More precisely, they prove
\begin{equation}\label{BN-CS}
\chi_\lambda(\zz)\prod_{\alpha>0} (1-q^{-1}\zz^{\alpha}) = \sum_{b\in\BB(\lambda+\rho)} G_\ii(b)q^{-\langle w_\circ(\wt(b) -\lambda -\rho),\rho\rangle} \zz^{w_\circ(\wt(b) - \rho)},
\end{equation}
where $\rho$ is the half-sum of the positive roots, $\wt(b)$ is the weight of $b$, and the coefficients $G_\ii(b)$ are defined using so-called BZL paths.
Since the definition of the coefficients $G_\ii(b)$ involves computing BZL paths, it is necessary to compute the entire crystal graph.

The crystal $\BB(\lambda +\rho)$ has a combinatorial realization using the set of semistandard Young tableaux due to Kashiwara and Nakashima (\cite{KN:94}). Hence each vertex of a crystal graph can be represented by a semistandard tableau. The goal of this paper is to define the coefficients in the sum of the Casselman-Shalika formula using only tableau data for each vertex of the crystal.  Indeed, for a tableau $b\in \BB(\lambda+\rho)$,
we define $\aa_{i,j}$ to be the number of $(j+1)$-colored boxes in rows $1$ through $i$ for $1 \leq i \leq j \leq r$, and  $\bb_{i,j}$ to be  the number of boxes in the $i$th row which have color greater or equal to $j+1$ for $1 \leq i \leq j \leq r$. Then we define the coefficients $C_{\lambda+\rho, q}(b)$ using these numbers $\aa_{i,j}$ and $\bb_{i,j}$, and prove \begin{equation}\label{eq:Cfunction}
\zz^\rho\chi_\lambda(\zz)\prod_{\alpha>0}(1-q^{-1}\zz^{-\alpha}) = \sum_{b\in \BB(\lambda+\rho)} C_{\lambda+\rho,q}(b)\zz^{\wt(b)}.
\end{equation}
Here the use of negative roots is only a slight modification for notational convenience; it is in accordance with the usual form of character formula for $\chi_\lambda(\zz)$.

The utilization of tableaux has benefits: we do not need to look at the crystal graph structure of $\BB(\lambda+ \rho)$ to calculate $C_{\lambda+\rho,q}(b)$ for each tableau $b$. Furthermore, the identification of tableaux with crystal elements enables us to easily exploit the compatibility property of crystals under tensor product and brings about new results.
Indeed, in the last section,  we consider how to combine our results in Section \ref{sec:tableaux}  with tensor products of crystals to obtain the sum of coefficients $C_{\lambda+\rho, q}(b)$ for a given weight $\lambda + \rho$. This sum of coefficients is the same as a $p$-part of Weyl group multiple Dirichlet series \cite{BBF:S,BBCFH:06,BBF:11a,BBF:11,CGco, CG:S} in the non-metaplectic case, and contains a lot of information on related representations as revealed in \cite{KL:11, KL:affine}. We will show this through concrete examples.

\subsection*{Acknowledgements}
At the beginning of this work, K.-H.\ L.\ benefited greatly from the Banff workshop on ``Whittaker Functions, Crystal Bases, and Quantum Groups'' in June 2010 and would like to thank the organizers---B.\ Brubaker, D.\ Bump, G.\ Chinta and P.\ Gunnells.
B.\ S.\ would like to thank all the developers of crystal support in Sage \cite{combinat,sage}.

\section{Crystals}\label{sec:crystals}

Let $r\geq 1$ and suppose $\mathfrak{g} = \mathfrak{sl}_{r+1}$ with simple roots $\{ \alpha_1,\dots,\alpha_r\}$, and let $I = \{1,\dots,r\}$. Let $P$ and $P^+$ denote the weight lattice and the set of dominant integral weights, respectively. Denote by $\Phi$ and $\Phi^+$, respectively, the set of roots and the set of positive roots. Let $\{h_1,\dots,h_r\}$ be the set of coroots and define a pairing $\langle \ ,\ \rangle \colon P^\vee\times P \longrightarrow \ZZ$ by $\langle h,\lambda \rangle = \lambda(h)$, where $P^\vee$ is the dual weight lattice. Let $\mathfrak{h} = \CC \otimes_\ZZ P^\vee$ be the Cartan subalgebra, and let $\mathfrak{h}^*$ be its dual. Denote the {\it Weyl vector} by $\rho$; this is the element $\rho\in \mathfrak{h}^*$ such that $\rho(h_i) = 1$ for all $i$.  The Weyl vector may also be defined as
\[
\rho = \frac12\sum_{\alpha >0} \alpha = \sum_{i=1}^r \omega_i,
\]
where $\omega_i$ is the $i$th fundamental weight.

Let $W$ be the Weyl group of $\mathfrak g$ with simple reflections $\{s_1,\dots,s_r\}$. To each reduced expression $w= s_{i_1}\dots s_{i_k}$ for $w\in W$, we associate a {\it reduced word}, which is defined to be the $k$-tuple of positive integers $\ii = (i_1,\dots,i_k)$, and denote the set of all reduced words $\ii$ of $w\in W$ by $R(w)$. In particular, we let $w_\circ$ be the longest element of the Weyl group and call $\ii = (i_1,\dots,i_N) \in R(w_\circ)$ a {\it long word}, where $N$ is the number of positive roots.

\medskip

A {\it $\mathfrak{g}$-crystal} is a set $\BB$ together with maps
\[
\widetilde e_i, \widetilde f_i\colon \BB \longrightarrow \BB\sqcup\{0\},\ \ \ \ \ \
\varepsilon_i,\varphi_i\colon \BB \longrightarrow \ZZ\sqcup\{-\infty\},\ \ \ \ \ \
\wt\colon \BB \longrightarrow P,
\]
that satisfy a certain set of axioms (see, e.g., \cite{HK:02}).
To each highest weight representation $V(\lambda)$ of $\mathfrak{g}$, there is an associated highest weight crystal $\BB(\lambda)$ which serves as a combinatorial frame of the representation $V(\lambda)$.  In particular, we can express the character $\chi_\lambda(\zz)$ of $V(\lambda)$ in terms of the highest weight crystal:
\[
\chi_\lambda(\zz) = \sum_{b\in\BB(\lambda)} \zz^{\wt(b)}.
\]

When working with tensor products, the combinatorial structure of $\mathcal{B}(\lambda)$ is advantageous because of the {\it tensor product rule} for crystals.
The tensor product rule determines the component of a tensor product on which the Kashiwara operators act. The {\it signature rule} is a systematic way of visualizing this procedure. Let $i\in I$ and set $\BB = \BB_1 \otimes \cdots \otimes \BB_m$. Take $b = b_1\otimes \cdots \otimes b_m \in \BB$. To determine the action of either $\widetilde e_i$ or $\widetilde f_i$ on $b$, create a sequence of $+$ and $-$ according to
\[
(\ \underbrace{-\cdots-}_{\varepsilon_i(b_1)},\underbrace{+\cdots+}_{\varphi_i(b_1)},
\cdots ,
\underbrace{-\cdots-}_{\varepsilon_i(b_m)},\underbrace{+\cdots+}_{\varphi_i(b_m)} \ )
\]
Cancel any $+-$ pair to obtain a sequence of $-$'s followed by $+$'s. We call the resulting sequence the {\it $i$-signature} of $b$, and denote it by $i$-sgn$(b)$. Then $\widetilde e_i$ acts on the component of $b$ corresponding to the rightmost $-$ in $i$-sgn$(b)$ and $\widetilde f_i$ acts on the component of $b$ corresponding to the leftmost $+$ in $i$-sgn$(b)$.

As an illustration, we apply this rule to the semistandard Young tableaux realization of $\mathfrak{sl}_{r+1}$-crystals $\BB(\lambda)$ of highest weight representations for $\lambda$ a dominant integral weight. This description is according to Kashiwara and Naka\-shima \cite{KN:94}.
For the fundamental weight $\omega_1$, the crystal graph of $\BB(\omega_1)$ is given by
\[
\BB(\omega_1): \ \ \ \
\begin{tikzpicture}[scale=1.5,baseline=-4]
 \node (1) at (0,0) {$\young(1)$};
 \node (2) at (1,0) {$\young(2)$};
 \node (d) at (2,0) {$\cdots$};
 \node (n-1) at (3,0) {$\young(r)$};
 \node (n) at (4.2,0) {$\boxed{r+1}$};
 \draw[->] (1) to node[above]{\tiny$1$} (2);
 \draw[->] (2) to node[above]{\tiny$2$} (d);
 \draw[->] (d) to node[above]{\tiny$r-1$} (n-1);
 \draw[->] (n-1) to node[above]{\tiny$r$} (n);
\end{tikzpicture}
\]
Using this fundamental crystal $\BB(\omega_1)$, we may understand any tableaux of shape $\lambda$ by embedding the corresponding crystal $\BB(\lambda)$ into
$\BB(\omega_1)^{\otimes m}$, where $m$ is the number of boxes in the $\lambda$ shape. For example, in type $A_4$, we have
\[
\BB(\omega_1+\omega_2+\omega_3) \ni b= \young(133,34,5) \mapsto
\young(3)\otimes\young(3)\otimes\young(4)\otimes\young(1)\otimes\young(3)\otimes\young(5)
\in \BB(\omega_1)^{\otimes 6}.
\]
On the image of $b$ through this embedding, we may apply the signature rule to determine which boxes $\widetilde f_i$ and $\widetilde e_i$ affect. In this case, with $i=3$, we have $3$-$\operatorname{sgn}(b) = (+,+,-,\cdot,+,\cdot) = (+,\cdot,\cdot,\cdot,+,\cdot)$. Thus $\widetilde e_3b = 0$ and
\[
\widetilde f_3b = \young(134,34,5).
\]

\medskip

For a given $\mathfrak{g}$-crystal $\BB$, long word $\ii =(i_1, i_2, \dots, i_N) \in R(w_\circ)$, and dominant $\lambda$, define the {\it BZL path} $\psi_\ii(b)$ of $b\in \BB$ as follows: Define $a_1$ to be the maximal integer such that $\widetilde e_{i_1}^{a_1} b \neq 0$. Then let $a_2$ be the maximal integer such that $\widetilde e_{i_2}^{a_2} \widetilde e_{i_1}^{a_1} b \neq 0$. Inductively, let $a_j$ be the maximal integer such that
\[
\widetilde e_{i_j}^{a_j} \widetilde e_{i_{j-1}}^{a_{j-1}} \cdots \widetilde e_{i_2}^{a_2} \widetilde e_{i_1}^{a_1} b \neq 0,
\]
for $j=1,\dots,N$.  Finally, we set $\psi_\ii(b) = (a_1,\dots,a_N)$.
The BZL paths are also known as {\it Kashiwara data} or {\it string parametrizations} in the literature (see, for example, \cite{BZ:01,Kam:07}). We may use these terms interchangeably.

Associated to each entry in a given BZL path is a decoration: a circle, a box, or both.
\begin{equation}\label{box-1}
\text{If $\widetilde f_{i_j} \widetilde e_{i_{j-1}}^{a_{j-1}} \cdots \widetilde e_{i_1}^{a_1}b = 0$, then box $a_j$.}\tag{B-I}
\end{equation}
To define the circling rule, we write the BZL paths in a triangular form:
\[
\begin{array}{c@{}c@{}c@{}c@{}c@{}c@{}c}
&&&a_1&&&\\
&&a_2&&a_3&&\\
&a_4&&a_5&&a_6&\\
\iddots&&\vdots&&\vdots&&\ddots
\end{array}
=
\begin{array}{c@{}c@{}c@{}c@{}c@{}c@{}c}
&&&a_{1,1}&&&\\
&&a_{2,1}&&a_{2,2}&&\\
&a_{3,1}&&a_{3,2}&&a_{3,3}&\\
\iddots&&\vdots&&\vdots&&\ddots
\end{array}
\]
We may now define the circling rule according to \cite{BBF:11,BN:10}.
\begin{equation}\label{circ-1}
\text{If $a_{i,j} = a_{i,j+1}$, then circle $a_{i,j}$.}\tag{C-I}
\end{equation}
We understand that the entries outside the triangle are zero, so the rightmost entry of a row is circled if it is zero. When representing these triangles in an inline form, we write
\[
\psi_\ii(b)=(a_{1,1};a_{2,1},a_{2,2};\dots;a_{r,1},\dots,a_{r,r}).
\]

In \cite{BBF:11a, BN:10}, Brubaker-Bump-Friedberg and Bump-Nakasuji define a function $G_\ii(b)$ on $\BB(\lambda+\rho)$ which allows them to write the Casselman-Shalika formula as a sum over $\BB(\lambda+\rho)$ using BZL paths.  From this point forward, we fix $\ii=(1,2,1,\dots,r,r-1,\dots,2,1)$.   Define
\[
G_\ii(b) = \prod_{a\in \psi_\ii(b)} \left\{\begin{array}{cl}
 q^a 	  & \text{if $a$ is circled but not boxed,}\\
 -q^{a-1}  & \text{if $a$ is boxed but not circled,}\\
 (q-1)q^{a-1} & \text{if $a$ is neither circled nor boxed,}\\
 0 	  & \text{if $a$ is both circled and boxed,}
\end{array}\right.
\]

\begin{prop}[{\cite{BN:10}}]\label{prop:BN}
We have
\[
\chi_\lambda(\zz)\prod_{\alpha >0} (1-q^{-1}\zz^{\alpha}) = \sum_{b\in \BB(\lambda+\rho)} G_\ii(b) q^{-\langle w_\circ(\wt(b)-\lambda-\rho),\rho\rangle}\zz^{w_\circ(\wt(b)-\rho)}.
\]
\end{prop}

\section{Using the tableaux model}\label{sec:tableaux}

From now on, we identify the set of semistandard Young tableaux of shape $\lambda$ for a dominant $\lambda$ with the crystal $\BB(\lambda)$ through the realization of Kashiwara-Nakashima \cite{KN:94}.

\begin{dfn}
Let $b\in \BB(\lambda+\rho)$ be a tableau.
\begin{enumerate}
\item Define $\aa_{i,j}$ to be the number of $(j+1)$-colored boxes in rows $1$ through $i$ for $1 \leq i \leq j \leq r$, and define the vector $\aa(b) \in \ZZ_{\ge0}^N$ by
\[
\aa(b)=(\aa_{1,1}, \aa_{1,2}, \ldots, \aa_{1,r}; \aa_{2,2}, \ldots , \aa_{2,r}; \ldots ; \aa_{r,r} ).
\]
\item  The number  $\bb_{i,j}$ is defined to be  the number of boxes in the $i$th row which have color greater or equal to $j+1$ for $1 \leq i \leq j \leq r$.
\item If a  tableaux $b\in \BB(\lambda+\rho)$ does not contain any box such that its entry is strictly greater than any entry in the next row, we say that $b$ is {\it strict}.
\end{enumerate}
\end{dfn}

\begin{ex}
Suppose $r=3$ and $\lambda = \omega_1+\omega_3$.  Consider the tableaux
\[
b_1 = \young(11124,223,34)
\ \ \text{ and }\ \
b_2 = \young(11223,233,34)\ .
\]
Then $b_1$ is not strict because the $4$ in the first row is larger than any entry in the second row.  However, $b_2$ is strict, and we have
\[
\aa(b_2) = (2,1,0;3,0;1).
\]
Moreover, $(\bb_{i,j}(b_2))_{1\leq i\leq j\leq r} = (3,1,0;2,0;1)$.
\end{ex}

\begin{dfn}
Let $b \in \BB(\lambda+\rho)$ be a tableaux.  We define a {\it $k$-segment} of $b$ (in the $i$th row) to be a maximal consecutive sequence of $k$-boxes in the $i$th row for any $i+1 \le k \le r+1$.
\end{dfn}

\begin{ex}
Suppose $r=3$ and $\lambda +\rho = 2\omega_1+2\omega_2 +\omega_3$, and consider the element
\[
b = \young(11234,233,4).
\]
Then there is a $2$-segment in the first row, a $3$-segment in each of the first and second rows, and a $4$-segment in each of the first and third rows.  Neither the $1$-boxes in the first row nor the $2$-box in the second row constitute a  $k$-segment.
\end{ex}

\begin{lemma}\label{lemma:stringbox}
Let $b\in \BB(\lambda+\rho)$. Then the sequences $\psi_{\ii}(b)=(a_{i,j})$ and $\aa(b)=(\aa_{i,j})$ are related via the formula
\[
a_{i,j} = \aa_{i-j+1,i}.
\]
\end{lemma}

\begin{proof}
We must show that $a_{i,j}$ is the number of $(i+1)$-colored boxes appearing in rows $1$ through $i-j+1$.  Fix $\ii = (1,2,1,3,2,1,\dots,r,r-1,\dots,2,1)$ as before, and let $b\in \BB(\lambda+\rho)$.  The action of the Kashiwara operator $\widetilde e_1$ is determined by the boxes colored $1$ and $2$.  In particular, we may have $1$-boxes in the first row and $2$-boxes in either the first or second row.  Using the semistandard condition on $b$, the $1$-signature of $b$ begins with a $-$ for every $2$-box in the first row.  Then any subsequent $-$ appearing in $1\text{-sgn}(b)$ comes from a $2$-box in the second row and is canceled by the $+$ coming from the mandatory $1$-box appearing above it in the first row.  Thus, $a_{1,1}$ is the number of $2$-boxes in the first row, so $a_{1,1}=\aa_{1,1}$. Note that there are no remaining $2$-boxes in the first row of $\widetilde e_1^{a_{1,1}}b$.

Next, the operator $\widetilde e_2$ must be applied to $b' = \widetilde e_1^{a_{1,1}}b$.  Here, the $2$-signature of $b'$ is completely determined by the $2$-boxes in the second row and the $3$-boxes in the first, second, and third rows.  Using a similar argument to the one from the previous paragraph, the semistandard condition on $b'$ implies that $a_{2,1}$ is the number of $3$-boxes in the first and second rows combined, so $a_{2,1}=\aa_{2,2}$.
For the next $\widetilde e_1$ operator, we apply the same argument as in the first paragraph to $b'' = \widetilde e_2^{a_{2,1}}b'$ and obtain that $a_{2,2}=\aa_{1,2}$. We can continue the process.

To make it sure that the process works, we consider the general case.  Suppose that we have \[ b''' = \widetilde e_{i-j+2}^{a_{i,j-1}} \cdots \widetilde e_1^{a_{2,2}}\widetilde e_2^{a_{2,1}}\widetilde e_1^{a_{1,1}} .\]  Extending the arguments from the last two paragraphs, any $(i+1)$-box in rows $1$ through $i-j+1$ of $b$ is demoted to an $(i-j+2)$-box in $b'''$, and any $(i+1)$-box in a row $k > i-j+1$ of $b$ is converted to a $k$-box in $b'''$.  So the $(i-j+1)$-signature of $b'''$ has the form
\[
(-,\dots,-,+,\dots,+),
\]
where each $-$ comes from an $(i-j+2)$-box in rows $1$ through $i-j+1$ and each $+$ comes from the $(i-j+1)$-boxes in the $(i-j+1)$st row.  Note that there is no cancellation $(+,-)$ as $b'''$ is semistandard.  Hence $a_{i,j}$ is exactly the number of $-$ appearing in the $(i-j+1)$-signature, which is precisely the number of $(i+1)$-boxes in rows $1$ through $(i-j+1)$ of $b$.
\end{proof}

\begin{remark} \label{rmk-removal}
The argument in the above proof shows how segments are removed from an arbitrary tableau $b$ in order starting with the (possible) $2$-segment in the first row.  In particular, we define $\widetilde e_{(i)} = \widetilde e_1^{a_{i,i}} \widetilde e_2^{a_{i,i-1}} \cdots \widetilde e_i^{a_{i,1}}$.  Then applying $\widetilde e_{(1)}$ to $b$ removes a $2$-segment.  Continuing to apply $\widetilde e_{(2)}$ to $\widetilde e_{(1)}b$ removes all $3$-segments (in the first and second rows), and so on.  This is a slight modification of a result shown in \cite{LS:11}.
\end{remark}

For $\lambda \in P^+$, write $\lambda+\rho$ as
\[
\lambda+\rho = (\ell_1 > \ell_2 > \cdots > \ell_r > \ell_{r+1}= 0),
\]
and define $\theta_i = \ell_i- \ell_{i+1}$ for $i=1,\dots,r$.  Let $\theta = (\theta_1,\dots,\theta_r)$.
For example, if $r=3$ and $\lambda=3\omega_1+\omega_2$, then
\[
\lambda+\rho = \yng(6,3,1),
\]
so $\theta = (3,2,1)$.

By Lemma \ref{lemma:stringbox}, we are able to reinterpret the boxing rule (B-I) and circling rule (C-I) in terms of Young tableaux. Recall that we have the data $\bb_{i,j}$ from the tableau $b$. Now we give a new definition of boxing and circling on the entries of $\aa(b)=(\aa_{i,j})$ for $b \in \BB(\lambda + \rho)$.
\begin{align}
&\text{Box $\aa_{i,j}$ if $\bb_{i,j} \geq \theta_i + \bb_{i+1,j+1}$.}\tag{B-II}\label{ourbox}\\
&\text{Circle $\aa_{i,j}$ if $\aa_{i,j} = \aa_{i-1,j}$.}\tag{C-II}\label{ourcirc}
\end{align}

\begin{prop}\label{prop:circle}
 An entry $a_{i,j}$ in $\psi_{\ii}(b)$ is circled by \eqref{circ-1} if and only if the corresponding entry in $\aa(b)$ is circled by \eqref{ourcirc}.
\end{prop}

\begin{proof}
By Lemma \ref{lemma:stringbox}, we have
\[
\aa_{i-j+1,i} = a_{i,j}   \quad \text{ and }\quad a_{i,j+1}= \aa_{i-(j+1)+1,i} = \aa_{i-j,i}.\qedhere
\] Now the result follows.
\end{proof}

The equivalence of boxing rules needs more care. We first prove a lemma.

\begin{lemma}\label{lemma:bbij}
Suppose $b\in \BB(\lambda+\rho)$ and write $\lambda+\rho = (\ell_1 > \ell_2 > \cdots > \ell_r > \ell_{r+1}= 0)$ as before.  Let $b' = \widetilde e_{i_{k-1}}^{a_{k-1}} \cdots \widetilde e_{i_1}^{a_1}b$ and assume $a_k = a_{j,j-i+1}=\aa_{i,j}$.  Then $\ell_i - \bb_{i,j}$ is equal to the number of $i$-boxes in the $i$th row of $b'$ and $\ell_{i+1} - \bb_{i+1, j+1}$ is equal to the number of $(i+1)$-boxes in the $(i+1)$st row of $b'$.
\end{lemma}

\begin{proof}
Since $a_k= a_{j,j-i+1}= \aa_{i,j}$, we have $i_k=i$ and the tableau $b'$ is in the middle of the process of removing $(j+1)$-segments from $b$. See Remark \ref{rmk-removal}. When all the $j$-segments have been removed, the number of $i$-colored boxes in the $i$th row is $\ell_i -\bb_{i,j}$. In the process of removing $(j+1)$-segments, until we get $b'$, the boxes that are colored $i+2$ through $j+1$ have been converted into $(i+1)$-colored boxes, and the number of $i$-colored boxes does not change. Thus the number of $i$-colored boxes in the $i$th row of $b'$ is $\ell_i -\bb_{i,j}$.

Now we look at the $(i+1)$st row of $b'$. When all the $j$-segments have been removed, there may be some $(j+1)$-boxes in the $(i+1)$st row. These boxes have been converted into $(i+1)$-boxes in the process of getting the tableau $b'$. Thus the number of $(i+1)$-boxes is given by $\ell_{i+1} - \bb_{i+1, j+1}$.
\end{proof}

\begin{prop}\label{prop:box}
An entry $a_{i,j}$ in $\psi_{\ii}(b)$ is boxed by \eqref{box-1} if and only if the corresponding entry in $\aa(b)$ is boxed by \eqref{ourbox}.
\end{prop}

\begin{proof}
Let $b' = \widetilde e_{i_{k-1}}^{a_{k-1}} \cdots \widetilde e_{i_1}^{a_1}b$ and assume $a_k = a_{j,j-i+1}= \aa_{i,j}$. Note that $i_k=i$. Then $\widetilde f_{i_k} b'=0$ if and only if the number of $i$-boxes in the $i$th row is less than or equal to the number of $(i+1)$-boxes in the $(i+1)$st row. That is, $\widetilde f_{i_k} b'=0$ if and only if
\[
\ell_i - \bb_{i,j} \le \ell_{i+1} - \bb_{i+1, j+1}, \text{ or equivalently, } \ \bb_{i,j} \ge \theta_i + \bb_{i+1, j+1} .\qedhere
\]
\end{proof}

Before stating the main theorem, we need one more lemma.

\begin{lemma}\label{lemma:notstrict}
If $b\in \BB(\lambda+\rho)$ is not strict, then $\aa(b)$ has an entry which is both circled and boxed.
\end{lemma}

\begin{proof}
Without loss of generality, suppose the first row of $b$ contains an entry strictly larger than any entry in the second row of $b$.  Suppose this large entry of the first row is $k+1\geq3$. Then $\aa_{1,k}=\aa_{2,k}$, so that $\aa_{2,k}$ is circled.  Note that $\theta_i \geq 1$ for all $i$ by the definition of $\lambda+\rho$.  Then $\bb_{2,k} = 0$ since $k+1$ is strictly greater than any entry in the second row.  Hence $\bb_{2,k}=0 \not\geq \theta_1 + \bb_{3,k+1}\geq 1$, and $\aa_{2,k}$ is boxed.
\end{proof}

For $b\in \BB(\lambda+\rho)$, we define $\non(b)$ to be the number of entries in $\aa(b)$ which are neither circled nor boxed, and define $\bbox(b)$ to be the number of entries in $\aa(b)$ which are boxed. Now define a function $C_{\lambda+\rho,q}$ on $\BB(\lambda+\rho)$ with values in $\ZZ[q^{-1}]$ by
\[
C_{\lambda+\rho,q}(b) =
\left\{\begin{array}{cl}
(-q^{-1})^{\bbox(b)}(1-q^{-1})^{\non(b)} & \text{if $b$ is strict},\\
0 & \text{otherwise}.
\end{array}\right.
\]

\begin{thm}\label{thm:CS-A}
We have
\begin{equation}\label{eq:CS-A}
\zz^\rho\chi_\lambda(\zz)\prod_{\alpha>0} (1-q^{-1}\zz^{-\alpha})
= \sum_{b\in \BB(\lambda+\rho)} C_{\lambda+\rho,q}(b) \zz^{\wt(b)}.
\end{equation}
\end{thm}

\begin{proof}

By Proposition \ref{prop:BN}, we have
\[
\chi_\lambda(\zz)\prod_{\alpha>0} (1-q^{-1}\zz^{\alpha}) = \sum_{b\in \BB(\lambda+\rho)} G_\ii(b) q^{-\langle w_\circ(\wt(b)-\lambda-\rho),\rho\rangle}\zz^{w_\circ(\wt(b)-\rho)}.
\]
We apply $w_\circ$ to both sides of the above equality. Recall that $\chi_\lambda(\zz)$ is invariant under the Weyl group action and that $w_\circ$ sends all the positive roots to negative roots. Thus we obtain
\begin{align*}
\chi_\lambda(\zz)\prod_{\alpha>0} (1-q^{-1}\zz^{-\alpha})
&= \sum_{b\in \BB(\lambda+\rho)} G_\ii(b) q^{-\langle w_\circ(\wt(b)-\lambda-\rho),\rho\rangle}\zz^{\wt(b)-\rho} \\
&= \sum_{b\in \BB(\lambda+\rho)} G_\ii(b) q^{\langle \wt(b)-\lambda-\rho,\rho\rangle}\zz^{\wt(b)-\rho} ,
\end{align*}
where we used the fact that $w_\circ (\rho) = - \rho$.

If $\wt(b)=\lambda +\rho - \sum_{i\in I} c_i \alpha_i$ for $b \in \BB(\lambda+\rho)$, we have
\[
\langle \wt(b)-\lambda-\rho,\rho \rangle = - \sum_{i\in I} c_i = - \sum_{a \in \psi_{\ii}(b)} a .
\]
Hence
\[
G_\ii(b)  q^{\langle \wt(b)-\lambda-\rho,\rho\rangle} = \prod_{a\in \psi_\ii(b)} \left\{\begin{array}{cl}
 1 	  & \text{if $a$ is circled but not boxed,}\\
 -q^{-1}  & \text{if $a$ is boxed but not circled,}\\
 1-q^{-1} & \text{if $a$ is neither circled nor boxed,}\\
 0 	  & \text{if $a$ is both circled and boxed,}
\end{array}\right.
\]
It follows from Propositions \ref{prop:circle}, \ref{prop:box} and Lemma \ref{lemma:notstrict} that
\[ G_\ii(b)  q^{\langle \wt(b)-\lambda-\rho,\rho\rangle} = C_{\lambda+\rho,q}(b) .\]
This completes the proof.
\end{proof}

\section{Deformed weight multiplicities}\label{sec:LR}

In \cite{KL:11, KL:affine}, a polynomial $H_{\lambda+\rho}(\mu;q)$ was defined and its representation theoretic meaning was investigated. This polynomial is the same as the $p$-part of Weyl group multiple Dirichlet series in the non-metaplectic case of type $A$. In this section, we show how to calculate the polynomial $H_{\lambda+\rho}(\mu;q)$ using our combinatorial results in the previous section and how to use the polynomial to obtain some information on the relevant representations.

\medskip

For $\lambda \in P^+$ and $\mu\in \bigoplus_{i\in I}\ZZ_{\ge0}\,\alpha_i$, we define a polynomial $H_{\lambda+\rho}(\mu; q) \in \ZZ[q^{-1}]$ by
\[
H_{\lambda+\rho}(\mu; q)=
\sum_{\substack{b\in\BB(\lambda+\rho) \\ \wt(b)=\lambda+\rho-\mu}} C_{\lambda+\rho,q}(b).
\]
We can also calculate $H_{\lambda+\rho}(\mu; q)$ using a tensor product of two crystals. It was also shown in \cite{KL:11} that
\[
H_{\lambda +\rho} (\mu; q)=\sum_{\substack{b'\otimes b\in\BB(\lambda)\otimes\BB(\rho)\\ \wt(b'\otimes b) =\lambda+\rho-\mu}} C_{\rho,q}(b) .
\]

\begin{ex} \label{example}
For $r=2$, we consider $\lambda= \omega_2$ and  $\mu=2\alpha_1+2\alpha_2$. The crystal $\BB(\omega_2+\rho)$ has exactly two tableau of weight $\omega_2+\rho-2\alpha_1-2\alpha_2$.
Namely,
\[
b_1=\young(122,33)
\text{ \ \ \ \ and \ \ \ \ }
b_2=\young(123,23).
\]
According to the boxing \eqref{ourbox} and circling rules \eqref{ourcirc},
\[
C_{\omega_2+\rho,q}(b_1) =-q^{-1}(1-q^{-1})
\text{ \ \ \ \ and \ \ \ \ }
C_{\omega_2+\rho,q}(b_2) =q^{-2}(1-q^{-1}).
\]
Thus we have
\begin{equation} \label{same}
H_{\omega_2+\rho}(\mu;q) = C_{\omega_2+\rho,q}(b_1)+C_{\omega_2+\rho,q}(b_2) = -q^{-1}(1-q^{-1})+q^{-2}(1-q^{-1}) .
\end{equation}

On the other hand, the crystal $\BB(\omega_2)\otimes\BB(\rho)$  has four tableau of weight $\omega_2+\rho-2\alpha_1-2\alpha_2$:
\[
b_3=\young(2,3)\otimes\young(12,3)\ ,
\ \
b_4=\young(2,3)\otimes\young(13,2)\ ,
\ \
b_5=\young(1,3)\otimes\young(22,3)\ ,
\ \
b_6=\young(1,2)\otimes\young(23,3)\ .
\]
Again, by the boxing \eqref{ourbox} and circling rules \eqref{ourcirc}, this time in $\BB(\rho)$, we have
\[
C_{\rho,q}(b_3)=-q^{-1}(1-q^{-1}),
\ \
C_{\rho,q}(b_4)=0,
\ \
C_{\rho,q}(b_5)=q^{-2},
\ \
C_{\rho,q}(b_6)=-q^{-3}.
\]
Thus we obtain
\begin{align*}
H_{\omega_2+\rho}(\mu;q)
&= C_{\rho,q}(b_3) + C_{\rho,q}(b_4) + C_{\rho,q}(b_5) + C_{\rho,q}(b_6) \\
&= -q^{-1}(1-q^{-1})+0+q^{-2}-q^{-3} \\
&= -q^{-1}(1-q^{-1})+q^{-2}(1-q^{-1}) ,
\end{align*}
which is the same as \eqref{same}.

We list the polynomials $H_{\omega_2+\rho}(\mu;q)$ for all the possible $\mu$; i.e., for the weights of $V(\omega_2+\rho)$.
\begin{equation}
\label{table}
\onehalfspacing
\begin{array}{c|c||c|c}
\mu & H_{\omega_2+\rho}(\mu) & \mu & H_{\omega_2+\rho}(\mu) \\\hline
0 & 1 & \alpha_1+2\alpha_2 & -2q^{-1}(1-q^{-1}) \\
\alpha_1 & -q^{-1} & \alpha_1+3\alpha_2 & q^{-2} \\
\alpha_2 & 1-q^{-1}& 2\alpha_1+2\alpha_2 & -q^{-1}(1-q^{-1})^2 \\
\alpha_1+\alpha_2 & (1-q^{-1})^2 & 2\alpha_1+3\alpha_2 & q^{-2}(1-q^{-1}) \\
2\alpha_2 & -q^{-1} & 3\alpha_1+2\alpha_2 & q^{-2} \\
2\alpha_1 + \alpha_2 & -q^{-1}(1-q^{-1}) & 3\alpha_1+3\alpha_2 & -q^{-3}
\end{array}
\end{equation}
\end{ex}

Assume that $\lambda \in P^+$. It was proved in \cite{KL:11,KL:affine} that the polynomial $H_{\lambda+\rho}(\mu;q)$ has the following properties.

\begin{enumerate}

\item The value $H_{\lambda+\rho}(\mu; \infty)$ is the multiplicity of the weight $\lambda-\mu$ in $V(\lambda)$.

\item The value $H_{\lambda+\rho}(\mu; -1)$ is the multiplicity of the weight $\lambda+\rho-\mu$ in the tensor product $V(\lambda)\otimes V(\rho)$.

\item  We have
\[
H_{\lambda +\rho}(\mu;1) =
\begin{cases}(-1)^{\ell(w)} &\text{ if } w \circ \lambda = \lambda - \mu \text{ for some } w \in W , \\
\quad 0 & \text{ otherwise,}
\end{cases}
\]
where we define $w \circ \lambda = w (\lambda + \rho)  -\rho$ for $w \in W$.

\end{enumerate}

\begin{ex}
Let us take $\lambda= \omega_2$ as in Example \ref{example}.
\begin{enumerate}
\item We put $q =\infty$, and see that $\omega_2$, $\omega_2 - \alpha_2$, and $\omega_2 - \alpha_1 -\alpha_2$ are weights of $V(\omega_2)$ with multiplicity $1$.
\item When $q=-1$, we see, for example, that the multiplicity of $\omega_2 +\rho - \alpha_1 -2\alpha_2$ is $4$ in the tensor product $V(\omega_2) \otimes V(\rho)$.
\item With the value $q=1$, we obtain that $\omega_2, \omega_2-\alpha_1$, $\omega_2-2\alpha_2$, $\omega_2- \alpha_1 -3 \alpha_2$, $\omega_2-3\alpha_1 - 2\alpha_2$, and $\omega_2-3\alpha_1 - 3\alpha_2$ are Weyl conjugates under the circle action.
\end{enumerate}
\end{ex}

\end{document}